\documentclass[psamsfonts]{amsart}
\usepackage[utf8]{inputenc}
\usepackage{amsfonts}
\usepackage{hyperref}
\hypersetup{
pdftitle={Lieb-Thirring inequality on Sphere},
pdfsubject={Mathematics, Lieb-Thirring inquality},
pdfauthor={Shihang Pan},
pdfkeywords={Lieb-Thirring inequalites, Sphere}
}
\usepackage{mathtools}
\usepackage{amsmath}
\usepackage{xcolor}
\usepackage{amsthm}
\usepackage{pdflscape}
\usepackage{pgfplots}
\usepackage{mathrsfs}
\usepackage{indentfirst}

\newtheorem{thm}{Theorem}[section]

\theoremstyle{definition}

\theoremstyle{remark}

\numberwithin{equation}{section}

\title{Lieb-Thirring inequality on the four-dimensional sphere and Torus}

\author{Shihang Pan }

\address{ Shihang Pan \newline \ \ \ School of Mathematical Sciences, Beijing Normal University, Beijing 100875, China}

\begin{document}

\renewcommand{\theequation}{\thesection.\arabic{equation}}

\begin{abstract}
In this paper, we mainly study the Lieb-Thirring inequality for families of orthonormal scalar functions  on the four-dimensional sphere $ \mathbb{S}^{4} $ and torus $ \mathbb{T}^{4} $. The bounds of all the constants involved are obtained. Specifically, we prove that the costant $ \mathcal{K}_{4}(\mathbb{S}^{4}) $ of the Lieb-Thirring inequality on the sphere $ \mathbb{S}^{4} $ satisfies
$$ 0.0844 \leq \mathcal{K}_{4}(\mathbb{S}^{4})\leq 0.1728,$$
and the constant $ \mathcal{K}_{4}(\mathbb{T}^{4}) $ of the Lieb-Thirring inequality on the torus $ \mathbb{T}^{4} $ satisfies
$$ 0.0190 \leq \mathcal{K}_{4}(\mathbb{T}^{4}) \leq   0.1222.$$
\noindent \textbf{e-mail:} \ \ \ \href{mailto:panshihang@mail.bnu.edu.cn}{panshihang@mail.bnu.edu.cn}\\
\end{abstract}

\maketitle
\tableofcontents

\section{Introduction}
Lieb-Thirring inequality is an important class of inequalities in quantum mechanics and semi-classical analysis.
It was first proposed by Lieb and Thirring \cite{LT76} to study the stability of matter in quantum mechanics.
Later, it had been applied to many areas of mathematics, such as, the infinite dimensional dynamical systems and  the attractors of the Navier-Stokes equations\cite{T97}.

Lieb-Thirring inequality gives a upper bounds for the sum of the $ \gamma-$powers of the negative eigenvalues $\{ \lambda_{n}(- \Delta + V) \}_{n}$ (from variational principle and min-max principle, we can set $ -\infty < \lambda_{1} \leq \lambda_{2} \leq ... $) of the  Schr\"odinger operator $H \doteq  -\Delta  + V   $ in $ \mathbb{R}^{d},$ namely
\begin{eqnarray}\label{S11}
\sum_{\lambda_{n} < 0}| \lambda_{n}(-\Delta + V) |^{\gamma} \leq \mathcal{L}_{\gamma, d}\int_{\mathbb{R}^{d}} V_{-}^{\gamma + \frac{d}{2}}(x) dx.
\end{eqnarray}
where $ \Delta $ denotes the Laplace operator in $ \mathbb{R}^{d} $, $ V $ is a real-value potential in $ \mathbb{R}^{d} $ with $ V \in L^{\gamma + \frac{d}{2}}(\mathbb{R}^{d}),$ $ V_{-}(x) = \max\{ 0, -V(x) \},$ and  $ \mathcal{L}_{ \gamma, d} $ is the finite non-negative best constant depended only on $ \gamma\ \text{and} \ d.$
In \cite{LT76}, Lieb-Thirring inequality (\ref{S11}) had been proved to hold in the following cases:
$$ d = 1, \gamma \geq \frac{1}{2},\ \ \ d = 2, \gamma > 0,\ \ \ \text{and}\ d\geq 3,\gamma \geq 0. $$
For the sharp constant $\mathcal{L}_{\gamma, d} $ in inequality (\ref{S11}), Lieb and Thirring \cite{LT76} conjectured that:
\begin{eqnarray}\label{LTconjecture}
\mathcal{L}_{\gamma, d} = \max\{\mathcal{L}_{\gamma, d}^{(1)}, \  \mathcal{L}_{\gamma, d}^{cl}   \},
\end{eqnarray}
where $\mathcal{L}_{\gamma, d}^{(1)}$ is so-called one-particle constant which is the optimal constant for the Keller's problem\cite{K61},
and $ \mathcal{L}_{\gamma, d}^{cl} $ is so-called semi-classical constant in Weyl's law:
$$ \lim_{\hbar \rightarrow 0} \hbar^{d}\sum_{n}|\lambda_{n}(-\hbar^{2}\Delta + V)|^{\gamma} = \mathcal{L}_{\gamma, d}^{cl}\int_{\mathbb{R}^{d}} V_{-}^{\gamma + \frac{d}{2}}(x) dx,$$
with
$$ \mathcal{L}_{\gamma, d}^{cl} = (4\pi)^{-\frac{d}{2}}\frac{\Gamma(\gamma + 1)}{\Gamma(\gamma + 1 + \frac{d}{2})}.$$
Inserting inequality (\ref{S11}) to the left hand side of the Weyl's laws, one can see $ \mathcal{L}_{\gamma, d} \geq  \mathcal{L}_{\gamma, d}^{cl}.$
When $ \gamma \geq \frac{3}{2} $ and $d \geq 1$, the Lieb-Thirring conjecture (\ref{LTconjecture}) had been proved  in \cite{LW00}, but for some special values of $\gamma$ and $d$, (\ref{LTconjecture}) had been proved to be failed, details can be found in \cite{F21}.

Here, we introduce some results about the best constants for the Lieb-Thirring inequality. In the original article \cite{LT76}, Lieb and Thirring proved that the best constant $  \mathcal{L}_{\gamma, d} = \mathcal{L}_{\gamma, d}^{cl} $  for $ \gamma =  \frac{2k + 1}{2}\ \text{in}\ k\in \mathbb{N}$ and $ d = 1 $.
For all $ \gamma \geq \frac{3}{2} $ in $ d = 1 $, Aizenman and Lieb \cite{AL78} obtained the similar result, namely  $  \mathcal{L}_{\gamma, 1} = \mathcal{L}_{\gamma, 1}^{cl}$.
Later, Laptev and Weidl proposed the 'lifting argument' and extended Aizenman-Lieb argument\cite{AL78} to high dimemsions \cite{LW00}, namely,  for all $\gamma \geq \frac{3}{2}$ and $ d \geq 1 $,   $ \mathcal{L}_{\gamma, d} = \mathcal{L}_{\gamma, d}^{cl} $.
When $ \gamma = \frac{1}{2}$ and $ d = 1$, the Lieb-Thirring conjecture (\ref{LTconjecture}) was proved to be true by Hundertmark, Lieb and Thomas \cite{HLT98}, namely $\mathcal{L}_{\frac{1}{2}, 1} = \mathcal{L}_{\frac{1}{2}, 1}^{(1)} = 2\mathcal{L}_{\frac{1}{2}, 1}^{cl}$.
By a standard dual argument, when $ \gamma = 1 $, inequality (\ref{S11}) is equivalent to the following inequality:
\begin{eqnarray}\label{S12}
\sum_{j = 1}^{N}\int_{\mathbb{R}^{d}}| \nabla u_{j} |^{2}dx \geq \mathcal{K}_{d}\int_{\mathbb{R}^{d}}\bigg(\sum_{j = 1}^{N}| u_{j} |^{2}  \bigg)^{1 + \frac{2}{d}}dx,\ \ d \geq 1,
\end{eqnarray}
where  $ \{u_{j}\}_{j = 1}^{N} \subseteq H^{1}(\mathbb{R}^{d}) $ is an arbitrary series of orthonormal functions in $ L^{2}(\mathbb{R}^{d}),$ and the corresponding constants $\mathcal{K}_{d}$ and $ \mathcal{L}_{1, d} $ satisfy
\begin{eqnarray}\label{S13}
\big( (1 + \frac{d}{2})\mathcal{L}_{1, d} \big)^{1 + \frac{2}{d}}\big(  (1 + \frac{2}{d})\mathcal{K}_{d}  \big)^{1 + \frac{d}{2}} = 1,
\end{eqnarray}
and the inequality (\ref{S12}) is also called Lieb-Thirring inequality.
By researching the equivalent inequality (\ref{S12}), there has been a lot of progress about the bounds on the best constant in the cases $ \gamma = 1$, details can be found in \cite{BS96, DLL08, EF91, HLW00, L84}.
But the cases $\gamma \in (\frac{1}{2}, \frac{3}{2}) $ in  $ d = 1 $, the conjecture (\ref{LTconjecture}) is still open.
The currently best bounds on the optimal constants $ \mathcal{L}_{\gamma, d} $ for $ \gamma \in (\frac{1}{2}, \frac{3}{2}) $  are
\begin{eqnarray*}
\mathcal{L}_{\gamma, d} \leq
\left\{ \begin{aligned}
&2.912\mathcal{L}_{\gamma, d}^{cl}\ \ \ \text{if}\ \frac{1}{2} \leq \gamma < 1, \  \text{and}\ \ d \geq 2 ,\\
&2\mathcal{L}_{\gamma, 1}^{cl}  \ \ \ \ \ \ \ \ \text{if}\ \frac{1}{2} \leq \gamma < 1, \ \text{and}\ \ d = 1,  \\
& 1.456\mathcal{L}_{\gamma, d}^{cl}  \ \ \  \text{if}\ 1 \leq \gamma < \frac{3}{2}.
\end{aligned}
\right.
\end{eqnarray*}
The first bounds follows from Laptev-Weidl lifting argument\cite{LW00},
the second bounds extends  Hundertmark's argument\cite{HLT98} to more general situation $ \gamma \in (\frac{1}{2}, 1)$ in $d = 1$,
and the third bounds was obtained by the optimal momentum decomposition and low momentum averaging from \cite{FHJN21}. When $ \gamma = 0$ in $ d \geq 3, $ Lieb-Thirring inequality (\ref{S11}) is so-called Cwikel-Lieb-Rozenblum (CLR) inequality:
\begin{eqnarray}\label{CLR}
N_{\leq}(- \Delta + V) \leq \mathcal{L}_{0, d}\int_{\mathbb{R}^{d}}V_{-}^{\frac{d}{2}}(x)dx,\ \ d \geq 3,
\end{eqnarray}
due to the works of Cwikel \cite{C77}, Lieb \cite{L76,L79} and Rozenblum \cite{R72,R76},
where $N_{\leq}(- \Delta + V)$ denotes the number of the negative eigenvalues of the Schr\"odinger operator, counting multiplicities.
When $ d = 3 $, Lieb \cite{L76,L79} proved that the  constant $ \mathcal{L}_{0, 3} $ satisfies
$$ \mathcal{L}_{0, 3} \leq 6.868924\mathcal{L}_{0, 3}^{cl}.$$

Analogue to $ \mathbb{R}^{d} $, the Lieb-Thirring inequality  on the manifolds $\mathbb{M}^{d} $ can be stated as
\begin{eqnarray}\label{LTonS}
\sum_{\nu_{n} \leq 0 }| \nu_{n} |^{\gamma} \leq \mathcal{L}_{\gamma, d}(\mathbb{M}^{d})\int_{\mathbb{M}^{d}} \big(V_{-}(x)\big)^{\gamma + \frac{d}{2}}dx,
\end{eqnarray}
where $\mathbb{M}^{d}$ is $d$-dimensional unit sphere $\mathbb{S}^{d}$ or torus $\mathbb{T}^{d},$
$ \{\nu_{n}\}_{n} $ is the negative eigenvalues  of the Schr\"odinger operator  $ - \Delta_{0} + V$( $-\infty < \nu_{1} \leq \nu_{2} \leq ... $), and $\Delta_{0}$ is the Laplace operator on $\mathbb{M}^{d}.$
 In \cite{I93}, Ilyin proved the inequality (\ref{LTonS}) by the method in \cite{LT76}, and applied it to study the attractors of Navier-Stokes equation on the sphere.
By standard dual argument,
when $ \gamma = 1 $, Lieb-Thirring inequality (\ref{LTonS}) also has the following equivalent forms:
\begin{eqnarray}\label{LTonSe}
 \int_{\mathbb{M}^{d}}\bigg(\sum_{j = 1}^{N}| u_{j} |^{2}  \bigg)^{1 + \frac{2}{d}}dx  \leq \mathcal{K}_{d}(\mathbb{M}^{d})\sum_{j = 1}^{N}\int_{\mathbb{M}^{d}}| \nabla_{0} u_{j} |^{2}dx ,\ \ d \geq 2,
\end{eqnarray}
where $\{ u_{j} \}_{j = 1}^{N} \subseteq \bar{H}^{1}(\mathbb{M}^{d}) \doteq {H}^{1}(\mathbb{M}^{d}) \cap \{ f| \int_{\mathbb{S}^{d}} f dS = 0  \} \subseteq {H}^{1}(\mathbb{M}^{d})$ is an  arbitrary  orthonormal series of scalar functions in
$ L^{2}(\mathbb{S}^{d}),$ $ \nabla_{0} $ denotes the  gradient operator on $ \mathbb{M}^{d} $,
and the corresponding constants $\mathcal{K}_{d}(\mathbb{M}^{d})$ and $\mathcal{L}_{\gamma, d}(\mathbb{M}^{d})$ also satisfy
$$\big( (1 + \frac{d}{2})\mathcal{L}_{1, d}(\mathbb{M}^{d}) \big)^{1 + \frac{2}{d}}\big(  (1 + \frac{2}{d})\mathcal{K}_{d}(\mathbb{M}^{d})  \big)^{1 + \frac{d}{2}} = 1.$$

By the Birman-Schwinger inequality (\cite{T97} Appendix 2.1),  Ilyin \cite{I12} directly proved the inequality (\ref{LTonS}) in $ \gamma = 1 $ and obtained the upper bounds for the best constants $ \mathcal{L}_{1, 2}(\mathbb{M}^{2}) $ and $ \mathcal{L}_{1, 3}(\mathbb{M}^{3})$,
namely, $ \mathcal{L}_{1, 2}(\mathbb{M}^{2}) < \frac{3}{8}$ and $ \mathcal{L}_{1, 3}(\mathbb{M}^{3}) \leq \delta_{\mathbb{M}^{3}}\frac{4}{15\pi}$ with $ \delta_{\mathbb{S}^{3}} = 1.0139\  \text{and}\ \delta_{\mathbb{T}^{3}} = 1$, where $\mathbb{M}^{2}\ \text{is}\ \mathbb{S}^{2}\ \text{or}\ \mathbb{T}^{2}$, $\mathbb{M}^{3}\ \text{is}\ \mathbb{S}^{3}\ \text{or}\ \mathbb{T}^{3}$.
In addition, 
on the two-dimensional sphere $\mathbb{S}^{2}$ and torus $\mathbb{T}^{2}$, it was shown in  \cite{IL19} that  the constants in inequality (\ref{LTonSe}) satisfy $ \mathcal{K}_{2}(\mathbb{S}^{2}) \leq \frac{3}{2\pi} $ and  $ \mathcal{K}_{ 2}(\mathbb{T}^{2}) \leq \frac{6}{\pi^{2}}$, respectively.
Later, in \cite{IL20},  by applying the method in \cite{FHJN21}  to the manifolds $ \mathbb{M}^{2}$, Ilyin, Laptev and Zelik improved the upper bounds for the constants $\mathcal{K}_{ 2}(\mathbb{M}^{2})$  in (\ref{LTonSe}) to $ \frac{3\pi}{32},$ namely, $\mathcal{K}_{2}(\mathbb{M}^{2}) \leq  \frac{3\pi}{32} (\mathbb{M}^{2} \ \text{denotes}\ \mathbb{S}^{2} \ \text{or}$ $ \mathbb{T}^{2} ).$
On the domain $ \Omega \subseteq \mathbb{S}^{2}$, Ilyin and Laptev \cite{IL19} proved that the Lieb-Thirring inequality has the following form:
\begin{eqnarray}\label{LTdomain}
\int_{\Omega}\bigg(\sum_{i = 1}^{N}| u_{i} |^{2}  \bigg)^{2}dx \leq \mathcal{K}(\Omega)\sum_{i = 1}^{N}\int_{\Omega}| \nabla_{0} u_{i} |^{2}dx,
\end{eqnarray}
where  $ \{ u_{j} \}_{j = 1}^{N} $ is an orthonormal family of scalar functions in $ H_{0}^{1}(\Omega),$ and the best constant satisfies
$$ \mathcal{K}(\Omega) \leq \frac{ 2 }{\pi} \frac{4\pi + |\Omega|}{4\pi - |\Omega|}, \ \ 0 < |\Omega| < |\mathbb{S}^{2}| =  4\pi.$$

So far there are no results about the bounds of the optimal constant of Lieb-Thirring inequality (\ref{LTonSe}) on  four-dimensional sphere $ \mathbb{S}^{4}$ and torus $ \mathbb{T}^{4}.$
Motivated by \cite{IL20, FHJN21}, in this paper, we shall study the Lieb-Thirring inequality on four-dimensional sphere
$ \mathbb{S}^{4} $ and torus $ \mathbb{T}^{4}$, and obtain the bounds of the optimal constants $\mathcal{K}_{4}(\mathbb{S}^{4})$ and $ \mathcal{K}_{4}(\mathbb{T}^{4})$.
The main results can be stated as follows:
\begin{thm}\label{thm0}
Let $u_{1}, u_{2}, ..., u_{N} \in H^{1}(\mathbb{M}^{4})\cap \{f | \int_{\mathbb{M}^{4}} f dx = 0, \} $ be a family othornormal  scalar functions in $ L^{2}(\mathbb{M}^{4})$, where $ \mathbb{M}^{4} $ denotes $ \mathbb{S}^{4} \ \text{or} \ \mathbb{T}^{4}$. Then
\begin{equation}\label{thmLT01}
\int_{\mathbb{M}^{4}}\big( \sum_{j = 1}^{N}|u_{j}(x)|^{2}  \big)^{\frac{3}{2}}dx \leq \mathcal{K}_{4}(\mathbb{M}^{4})\sum_{j = 1}^{N}\int_{\mathbb{M}^{4}}|\nabla u_{j} |^{2}dx,
\end{equation}
where the best constants $\mathcal{K}_{4}(\mathbb{M}^{4})$ satisfy
\begin{equation}\label{thmLT02}
 0.0844 \leq \mathcal{K}_{4}(\mathbb{S}^{4}) \leq 0.1728,
\end{equation}
and
\begin{equation}\label{thmLT03}
 0.0190 \leq \mathcal{K}_{4}(\mathbb{T}^{4}) \leq 0.1222.
\end{equation}
\end{thm}

For convenience, we will divide the proof of Theorem 1.1 into two part, namely,  Theorem 2.1 and Theorem 3.1.
In section 2,  we mainly prove the Theorem 2.1, namely, to prove the constant $ \mathcal{K}_{4}(\mathbb{S}^{4})$ for Lieb-Thirring inequality on the four-dimensional sphere $\mathbb{S}^{4}$ satisfies (\ref{thmLT02}).
In section $3$, by applying the method on the sphere to torus, and combining the Possion summation formula with Bessel's inequality,  we will prove the Theorem 3.1, namely, to prove  the constant $ \mathcal{K}_{4}(\mathbb{T}^{4})$ for Lieb-Thirring  inequality on the four-dimensional torus satisfies (\ref{thmLT03}).

\section{Lieb-Thirring Inequality on $\mathbb{S}^{4}$}

In this section, we focus on the constant of Lieb-Thirring inequality on four dimensional sphere $\mathbb{S}^{4}$. We first introduce some notations and known results about the Laplace operator on  $d$-dimensional sphere $\mathbb{S}^{d}$($\mathbb{S}^{d} $ is the unit sphere sphere embedded in Euclidean space $ \mathbb{R}^{d + 1} $). The Laplace operator on sphere is  so-called Laplace-Beltrami operator.
For convenience, we denote the Laplace-Beltrami operator on $\mathbb{S}^{d}$ by $\Delta_{\mathbb{S}^{d}}$.
Follows  \cite{SW72}, Laplace-Beltrami operator $\Delta_{\mathbb{S}^{d}}$ can be written as  the product of two 'gradient operator' on $d$-dimensional sphere $ \nabla_{d} $,
namely, $\Delta_{\mathbb{S}^{d}} = \nabla_{d} \cdot \nabla_{d}.$
From \cite{SW72}, the eigenvalues $\Lambda_{n}$ and the corresponding  eigenfunctions $ Y_{n}^{k} $ of $-\Delta_{\mathbb{S}^{d}}$  satisfy
$$ -\Delta_{\mathbb{S}^{d}}Y_{n}^{k} = \Lambda_{n} Y_{n}^{k},\ n = 0, 1, 2, ... \ , k = 1, 2, ..., \alpha_{n},$$
where $\Lambda_{n} = n( n + d - 1)$, $ \alpha_{n} = C_{d + n}^{n} - C_{d + n - 2}^{n - 2} $ is multiplicity of the eigenvuale $\Lambda_{n},$ $ Y_{n}^{k} $ is the spherical harmonic, and the series  $ \cup_{n = 1}^{\infty}\{ Y_{n}^{k} \}_{k = 1}^{\alpha_{n}} $ is an orthonormal baises of $ L^{2}(\mathbb{S}^{d}).$
Moreover,  for any $x \in \mathbb{S}^{d},$ the sum of squares of the eigenfunctions $Y_{n}^{k}(x)$ satisfies:
$$ \sum_{k = 1}^{\alpha_{n}} (Y_{n}^{k}(x))^{2} = \frac{\alpha_{n}}{\omega_{d}} , $$
where $ \omega_{d} $ is the area of $d$-dimensional unit sphere.

The following theorem is about the Lieb-Thirring inequalities on the four dimension sphere $\mathbb{S}^{4}$, namely, the Theorem $1.1$ for $ \mathbb{S}^{4}$.
For simplicity, we denote the Laplace-Beltrami operator and the corresponding 'gradient operator' on $\mathbb{S}^{4}$ by $-\Delta_{\mathbb{S}^{4}}$ and $ \nabla_{0}$, and $\Delta_{\mathbb{S}^{4}} = \nabla_{0} \cdot \nabla_{0}.$

\begin{thm}\label{thm1}
Let $u_{1}, u_{2}, ..., u_{N} \in H^{1}(\mathbb{S}^{4})\cap \{f | \int_{\mathbb{S}^{4}} f dx = 0, \} $ be a family othornormal  scalar functions in $ L^{2}(\mathbb{S}^{4})$. Then
\begin{equation}\label{thmLT1}
\int_{\mathbb{S}^{4}}\big( \sum_{j = 1}^{N}|u_{j}(x)|^{2}  \big)^{\frac{3}{2}}dx \leq \mathcal{K}_{4}(\mathbb{S}^{4})\sum_{j = 1}^{N}\int_{\mathbb{S}^{4}}|\nabla_{0} u_{j} |^{2}dx,
\end{equation}
where the constant $\mathcal{K}_{4}$ satisfies
\begin{equation}\label{thmLT2}
 \frac{3}{8\sqrt{2}\pi} \leq \mathcal{K}_{4}(\mathbb{S}^{4}) \leq \frac{\sqrt{2 B(\frac{2}{3}, \frac{4}{3})}}{9}.
\end{equation}
\end{thm}
\begin{proof}
First, we shall prove that the constant $\mathcal{K}_{4}(\mathbb{S}^{4})$ in the inequality (\ref{thmLT1}) satisfies the left side of inequality (\ref{thmLT2}), namely,
$$ \mathcal{K}_{4}(\mathbb{S}^{4}) \geq \frac{3}{8\sqrt{2}\pi}. $$
Let the orthonormal series $ \{u_{j}\} $ in this theorem be the spherical harmonic series $ \{Y_{n}^{k}\} $ on $\mathbb{S}^{4}$ ($ d = 4 $) where $n = 1, 2, ..., M-1,$ and  $k = 1, ...,\alpha_{n}.$
Based on the properties of spherical harmonic series $ \{Y_{n}^{k}\} $, we obtain
\begin{eqnarray}\label{S22}
\begin{split}
\int_{\mathbb{S}^{4}} \sum_{n = 1}^{M - 1}\sum_{k = 1}^{\alpha_{n}} (Y_{n}^{k}(x))^{2} dx &= \sum_{n = 1}^{M - 1}\sum_{k = 1}^{\alpha_{n}} \int_{\mathbb{S}^{4}} (Y_{n}^{k}(x))^{2} dx\\
&= \sum_{n = 1}^{M - 1} \big(C_{n + 4}^{n} - C_{n + 2}^{2} \big) \\
&= \sum_{n = 1}^{M - 1} \frac{(2n + 3)(n + 2)(n + 1)}{3!} \\
&\doteq \frac{1}{6}P(M),
\end{split}
\end{eqnarray}
where $P(M) = \frac{1}{2}M^{4} + C_{1} M^{3} + C_{2} M^{2} + C_{3}M + C_{4} $ and $C_{i}$($i = 1, 2, 3, 4$) are finite positive constant.
Applying H\"{o}lder inequality to the left side of (\ref{S22}), so that
\begin{eqnarray}\label{S23}
\begin{split}
\int_{\mathbb{S}^{4}} \sum_{n = 1}^{M - 1}\sum_{k = 1}^{\alpha_{n}} (Y_{n}^{k}(x))^{2} dx \leq \omega_{4}^{\frac{1}{3}} \bigg(\int_{\mathbb{S}^{4}} \big( \sum_{n = 1}^{M - 1}\sum_{k = 1}^{\alpha_{n}} (Y_{n}^{k}(x))^{2}\big)^{\frac{3}{2}} dx\bigg)^{\frac{2}{3}}.
\end{split}
\end{eqnarray}
Since $ \{Y_{n}^{k}\} $ is an orthonormal series, and inserting the series $ \{Y_{n}^{k}\} $ into the inequality (\ref{thmLT1}),  which gives
\begin{eqnarray}\label{S24}
\begin{split}
\int_{\mathbb{S}^{4}} \big( \sum_{n = 1}^{M - 1}\sum_{k = 1}^{\alpha_{n}} (Y_{n}^{k}(x))^{2}\big)^{\frac{3}{2}} dx &\leq \mathcal{K}_{4}(\mathbb{S}^{4}) \sum_{n = 1}^{M - 1}\sum_{k = 1}^{\alpha_{n}} \int_{\mathbb{S}^{4}}  | \nabla_{0}Y_{n}^{k}(x)|^{2} dx \\
&= \mathcal{K}_{4}(\mathbb{S}^{4}) \sum_{n = 1}^{M - 1}\sum_{k = 1}^{\alpha_{n}} \int_{\mathbb{S}^{4}}\big(-\Delta_{\mathbb{S}^{4}}Y_{n}^{k}(x)\big) Y_{n}^{k}(x) dx\\
&= \mathcal{K}_{4}(\mathbb{S}^{4}) \sum_{n = 1}^{M - 1}\sum_{k = 1}^{\alpha_{n}} \int_{\mathbb{S}^{4}}n(n + 3)  \big(Y_{n}^{k}(x)\big)^{2} dx \\
&= \mathcal{K}_{4}(\mathbb{S}^{4})\sum_{n = 1}^{M - 1} n(n + 3)\big( C_{n + 4}^{n} - C_{n + 2}^{2} \big) \\
&= \mathcal{K}_{4}(\mathbb{S}^{4})\sum_{n = 1}^{M - 1} \frac{ (2n  + 3)(n + 3)(n + 2)(n + 1)n }{3!}\\
&\doteq \frac{\mathcal{K}_{4}(\mathbb{S}^{4})}{6} Q(M),
\end{split}
\end{eqnarray}
where $Q(M) = \frac{1}{3} M^{6} + C_{1}M^{5} + C_{2}M^{4} + C_{3}M^{3} + C_{4}M^{2} + C_{5}M + C_{6}$ and $ C_{i}$($i = 1,...,6$) are finite positive constant.
Combining (\ref{S24}), (\ref{S22}) with (\ref{S23}), we have
\begin{eqnarray}\label{S25}
\begin{split}
\frac{1}{6}P(M) &= \int_{\mathbb{S}^{4}} \sum_{n = 1}^{M - 1}\sum_{k = 1}^{\alpha_{n}} (Y_{n}^{k}(x))^{2} dx\\
&\leq \omega_{4}^{\frac{1}{3}} \bigg(\int_{\mathbb{S}^{4}} \big( \sum_{n = 1}^{M - 1}\sum_{k = 1}^{\alpha_{n}} (Y_{n}^{k}(x))^{2}\big)^{\frac{3}{2}} dx\bigg)^{\frac{2}{3}}\\
&= \omega_{4}^{\frac{1}{3}} \big(  \frac{\mathcal{K}_{4}(\mathbb{S}^{4})}{6} Q(M)  \big)^{\frac{2}{3}}.
\end{split}
\end{eqnarray}
Cubing both sides of the inequality (\ref{S25}), so that
\begin{equation}\label{S26}
 \big( \frac{P(M)}{6} \big)^{3} \leq \omega_{4}\big(  \frac{\mathcal{K}_{4}(\mathbb{S}^{4})}{6} Q(M)  \big)^{2}.
\end{equation}
Then, dividing both sides of the inequality (\ref{S26}) by $ \omega_{4}\big(  \frac{Q(M)}{6} \big)^{2},$ and let $M$ tends to $\infty$, we have
\begin{eqnarray}\label{S27}
\begin{split}
\mathcal{K}_{4}^{2}(\mathbb{S}^{4}) \geq \frac{ P(M)^{3} }{ 6 \omega_{4} Q(M)^{2} }
= \frac{9}{128\pi^{2}} (M \rightarrow \infty).
\end{split}
\end{eqnarray}
Thus, we obtain the lower bounded of the constant $ \mathcal{K}_{4}(\mathbb{S}^{4})$
$$ \mathcal{K}_{4}(\mathbb{S}^{4}) \geq \frac{3}{8\sqrt{2}\pi} = 0.0844.$$

Next, we turn to prove the right side of (\ref{thmLT2}).
Let $g$ be a smooth nonnegative function and satisfies its square integral  on $\mathbb{R}^{+}$ equals to $1$, namely,
\begin{equation}\label{S21}
\int_{0}^{+\infty}(g(t))^{2}dt = 1.
\end{equation}
In fact, we know that the eigenfunction series $\{ Y_{n}^{k} \}_{n, k}$ of $ -\Delta_{\mathbb{S}^{4}} $ is a othornormal basis of  $ H^{1}(\mathbb{S}^{4}),$
\begin{eqnarray}\label{S29}
\int_{\mathbb{S}^{4}} Y_{n}^{k} Y_{m}^{l} dx =
\left\{ \begin{aligned}
1&  \ \ \  \text{when}\ n = m\ \text{and}\ k = l,\\
0&  \ \ \  \text{otherwise.}
\end{aligned}
\right.
\end{eqnarray}
Therefore,  for any function $ u(x) \in \{u | \int_{\mathbb{S}^{d}} u dx = 0, u \in H^{1}(\mathbb{S}^{d})  \} $, which can be expanded in the follows:
$$ u_{n}^{k} = \int_{\mathbb{S}^{4}} u(x) Y_{n}^{k}(x)dx = \big(u(x), Y_{n}^{k}(x) \big), \ \  u(x) = \sum_{n = 1}^{\infty} \sum_{k = 1}^{\alpha_{n}} u_{n}^{k} Y_{n}^{k}(x). $$
Based on the above expanded expression of $u(x)$ and (\ref{S21}), we have
\begin{eqnarray}\label{S28}
\begin{split}
\int_{\mathbb{S}^{4}}|\nabla_{0} u(x) |^{2}dx &= \int_{\mathbb{S}^{4}}\big(-\Delta_{\mathbb{S}^{4}} \sum_{n = 1}^{\infty} \sum_{k = 1}^{\alpha_{n}} u_{n}^{k} Y_{n}^{k}(x)  \big)\big( \sum_{n = 1}^{\infty} \sum_{k = 1}^{\alpha_{n}} u_{n}^{k} Y_{n}^{k}(x) \big)dx \\
&= \int_{\mathbb{S}^{4}}\big( \sum_{n = 1}^{\infty} n(n + 3) \sum_{k = 1}^{\alpha_{n}} u_{n}^{k} Y_{n}^{k}(x)  \big)\big( \sum_{n = 1}^{\infty} \sum_{k = 1}^{\alpha_{n}} u_{n}^{k} Y_{n}^{k}(x) \big)dx \\
&= \int_{\mathbb{S}^{4}}\big( \sum_{n = 1}^{\infty} \int_{0}^{\infty}g( \frac{E}{n(n + 3)})^{2}dE \sum_{k = 1}^{\alpha_{n}} u_{n}^{k} Y_{n}^{k}(x)  \big)\big( \sum_{n = 1}^{\infty} \sum_{k = 1}^{\alpha_{n}} u_{n}^{k} Y_{n}^{k}(x) \big)dx\\
&= \int_{\mathbb{S}^{4}}\int_{0}^{\infty}\big( \sum_{n = 1}^{\infty} g( \frac{E}{n(n + 3)}) \sum_{k = 1}^{\alpha_{n}} u_{n}^{k} Y_{n}^{k}(x)  \big)^{2} dE dx\\
&\doteq \int_{\mathbb{S}^{4}}\int_{0}^{\infty} |u^{E}(x)|^{2}dE dx,
\end{split}
\end{eqnarray}
where
$$ u^{E}(x) = \sum_{n = 1}^{\infty} g( \frac{E}{n(n + 3)}) \sum_{k = 1}^{\alpha_{n}} u_{n}^{k} Y_{n}^{k}(x).  $$
Similarly, for any $u_{j}$ in the sequence $\{ u_{j} \}_{j = 1}^{N}$, we can introduce the following  expression:
\begin{eqnarray}\label{S211}
u_{j}^{E}(x) = \sum_{n = 1}^{\infty} g( \frac{E}{n(n + 3)}) \sum_{k = 1}^{\alpha_{n}} u_{j,n}^{k} Y_{n}^{k}(x),
\end{eqnarray}
where
$$ u_{j,n}^{k} = \int_{\mathbb{S}^{4}} u_{j}(x)Y_{n}^{k}(x) dx = \big( u_{j}(x), Y_{n}^{k}(x)  \big), $$
and for any $u_{j}$ in $ \{ u_{j} \}_{j = 1}^{N} $,
\begin{eqnarray}\label{S212}
u_{j}(x) = \sum_{n = 1}^{\infty}\sum_{k = 1}^{\alpha_{n}} u_{j,n}^{k} Y_{n}^{k}(x).
\end{eqnarray}
We now turn to $ \sum_{j = 1}^{N}|u_{j}(x)|^{2} $, for any $ \epsilon > 0 $, and combine with Cauchy inequality, which yields
\begin{eqnarray}\label{S210}
\begin{split}
\sum_{j = 1}^{N}|u_{j}(x)|^{2} &= \sum_{j = 1}^{N}|u_{j}(x) - u_{j}^{E}(x) + u_{j}^{E}(x) |^{2}\\
&= \sum_{j = 1}^{N}|u_{j}^{E}(x)|^{2} + 2\sum_{j = 1}^{N}u_{j}^{E}(x)\big( u_{j}(x) - u_{j}^{E}(x)  \big) + \sum_{j = 1}^{N}| u_{j}(x) - u_{j}^{E}(x) |^{2}\\
&\leq (1 + \epsilon) \sum_{j = 1}^{N}|u_{j}^{E}(x)|^{2} + (1 + \frac{1}{\epsilon})\sum_{j = 1}^{N}| u_{j}(x) - u_{j}^{E}(x) |^{2}.
\end{split}
\end{eqnarray}
Applying (\ref{S211}) and (\ref{S212}) to the last term of the right side of (\ref{S210}),
so that
\begin{eqnarray}\label{S213}
\begin{split}
\sum_{j = 1}^{N}| u_{j}(x) - u_{j}^{E}(x) |^{2} &= \sum_{j = 1}^{N}\mid \sum_{n = 1}^{\infty}\big(1 -  g( \frac{E}{n(n + 3)}) \big)\sum_{k = 1}^{\alpha_{n}} u_{j,n}^{k} Y_{n}^{k}(x) \mid^{2}\\
&= \sum_{j = 1}^{N}\mid \sum_{n = 1}^{\infty}\big(1 -  g( \frac{E}{n(n + 3)}) \big)\sum_{k = 1}^{\alpha_{n}} \big( u_{j}(x'), Y_{n}^{k}(x')  \big) Y_{n}^{k}(x) \mid^{2}\\
&= \sum_{j = 1}^{N}\mid \bigg( u_{j}(x'),  \sum_{n = 1}^{\infty}\big(1 -  g( \frac{E}{n(n + 3)}) \big)\sum_{k = 1}^{\alpha_{n}} Y_{n}^{k}(x')Y_{n}^{k}(x)  \bigg)  \mid^{2}\\
&\doteq \sum_{j = 1}^{N}\mid \bigg( u_{j}(x'), \Psi^{E}(x', x)  \bigg) \mid^{2},
\end{split}
\end{eqnarray}
where
$$ \Psi^{E}(x', x) = \sum_{n = 1}^{\infty}\big(1 -  g( \frac{E}{n(n + 3)}) \big)\sum_{k = 1}^{\alpha_{n}} Y_{n}^{k}(x')Y_{n}^{k}(x). $$
Combining with the Bessel's inequality, we have the following estimate for (\ref{S213})
\begin{eqnarray}\label{S214}
\begin{split}
\sum_{j = 1}^{N}| u_{j}(x) - u_{j}^{E}(x) |^{2} & = \sum_{j = 1}^{N}\mid \bigg( u_{j}(x'), \Psi^{E}(x', x)  \bigg) \mid^{2} \\
&\leq  \int_{\mathbb{S}^{4}} \big( \Psi^{E}(x', x)  \big)^{2} d x'\\
&= \sum_{n = 1}^{\infty}\big(1 -  g( \frac{E}{n(n + 3)}) \big)^{2}\sum_{k = 1}^{\alpha_{n}} \big(Y_{n}^{k}(x)\big)^{2} \\
&= \sum_{n = 1}^{\infty}\big(1 -  g( \frac{E}{n(n + 3)}) \big)^{2}\frac{\alpha_{n}}{\omega_{4}}\\
&= \frac{1}{6\omega_{4}}\sum_{n = 1}^{\infty}\big(1 -  g( \frac{E}{n(n + 3)}) \big)^{2}(2n + 3)(n(n + 3) + 2).
\end{split}
\end{eqnarray}
The next main work is to estimate the right side of (\ref{S214}), which is a key point in the our proof.
Before that, we need a determined function $g$ under the condition
$$ \int_{0}^{+\infty}(g(t))^{2}dt = 1. $$
According to the Lemma $ 5 $ in \cite{FHJN21}, we can obtain the expression of $g$
\begin{eqnarray}\label{S215}
g(t) = \frac{1}{1 + {(\rho t)}^{3}},
\end{eqnarray}
where $ \rho  =   \frac{4\pi}{9\sqrt{3}}.$
Denote $ \nu = \frac{1}{\rho E} $ and let $ G(t) = \frac{t}{(1 + t^{3})^{2}}.$
Then, we obtain the estimate for the  the right side of (\ref{S214}),
\begin{eqnarray}\label{S216}
\begin{split}
\sum_{n = 1}^{\infty}\big(1 -  g( \frac{E}{n(n + 3)}) \big)^{2}&(2n + 3)\big(n(n + 3) + 2\big)\\
 &= \sum_{n = 1}^{\infty} \frac{(2n + 3)\big(n(n + 3) +2\big)}{\Big(  1 + \big(  \frac{n(n + 3)}{\rho E}  \big)^{3} \Big)^{2}}\\
&\leq \sum_{n = 1}^{\infty} \frac{(2n + 3)\big(2n(n + 3)\big)}{\Big(  1 + \big(  \frac{n(n + 3)}{\rho E}  \big)^{3}\Big)^{2}}\\
&= \sum_{n = 1}^{\infty}\frac{2(2n + 3)G\big(\nu n(n + 3)\big)}{\nu}\\
&\doteq \sum_{n = 1}^{\infty}\frac{2R(n)}{\nu},
\end{split}
\end{eqnarray}
where $$ R(n) = (2n + 3)G\big(\nu n(n + 3)\big),$$and
$$ R(x) = (2x + 3)G\big(\nu x(x + 3)\big),$$ with $R(0) = 0 $,
and $R(x)$ is differentiable  respect to $x$.
By the Euler-Maclaurin formula,
\begin{eqnarray*}\label{S217}
\sum_{n = 0}^{\infty}R(n) = \int_{0}^{\infty}R(x)dx + \frac{1}{2}R(0) - \sum_{j = 2}^{k - 1}\frac{B_{j}}{j!}R^{(j - 1)}(0) - \int_{0}^{\infty}\frac{B_{k}(x)}{k!}R^{k}(x) dx,
\end{eqnarray*}
where $B_{j}$ are the Bernoulli numbers: $ B_{2} = \frac{1}{6}, \ B_{3} = 0, \ B_{3} =- \frac{1}{30}, B_{5} = 0, B_{6} = \frac{1}{42},...,$
and $ B_{k}(x) $ are the periodic Bernoulli polynomials. Since $ G'(0) = 1, G''(0) = G'''(0) = G^{(4)}(0) = G^{(5)}(0) = 0,$ we have
$ R'(0) = 9 \nu,\ R''(0) = 18 \nu, \ R'''(0) = 12 \nu,\ R^{(4)}(0) = R^{(5)}(0) = 0 .$
Taking $k = 6$ and $ \nu \rightarrow 0 ( E \rightarrow 0), $ we have
\begin{eqnarray}\label{S218}
\begin{split}
\sum_{n = 0}^{\infty}R(n) &= \int_{0}^{\infty}R(x)dx - \frac{11}{25} \nu + O(\nu^{2})\ \ \  (\text{as}\ \nu \rightarrow 0 )\\
&\leq  \int_{0}^{\infty}R(x)dx = \int_{0}^{\infty} (2x + 3)G\big(\nu x(x + 3)\big)dx\\
&= \frac{1}{\nu} \int_{0}^{\infty}G(t)dt\\
&= \frac{B(\frac{2}{3}, \frac{4}{3})}{3 \nu},
\end{split}
\end{eqnarray}
where $B(\frac{2}{3}, \frac{4}{3})$ is Beta function.
Therefore, there exists sufficiently large $ E_{0} $ ($ 0 < E_{0} < \infty $) such that
\begin{eqnarray}\label{S219}
 \sum_{n = 1}^{\infty} \frac{(2n + 3)\big(2n(n + 3)\big)}{\Big(  1 + \big(  \frac{n(n + 3)}{\rho E}  \big)^{3}\Big)^{2}} \leq \frac{2\rho^{2}E^{2}B(\frac{2}{3}, \frac{4}{3})}{3}
\end{eqnarray}
holds for any $ E \geq E_{0} $.
On the other hand, on  the finite interval $E \in [0, E_{0} ],$ by the definition of definite integral,
\begin{eqnarray}\label{S2161}
\begin{split}
\sum_{n = 1}^{\infty} \frac{(2n + 3)\big(n(n + 3) +2\big)}{\Big(  1 + \big(  \frac{n(n + 3)}{\rho E}  \big)^{3} \Big)^{2}}
&\leq \frac{(2 + 3)\big((1 + 3) + 2\big)}{\Big(  1 + \big(  \frac{1(1 + 3)}{\rho E}  \big)^{3}\Big)^{2}} + \sum_{n = 2}^{\infty} \frac{(2n + 3)\big(2n(n + 3)\big)}{\Big(  1 + \big(  \frac{n(n + 3)}{\rho E}  \big)^{3}\Big)^{2}}\\
&= \frac{30}{\big( 1 + \frac{64}{\rho^{3} E^{3}} \big)^{2}} + \sum_{n = 2}^{\infty} \frac{(4n + 6)\big(n(n + 3)\big)}{\Big(  1 + \big(  \frac{n(n + 3)}{\rho E}  \big)^{3}\Big)^{2}} \\
&\leq \frac{30}{\big( 1 + \frac{64}{\rho^{3} E^{3}} \big)^{2}} + \sum_{n = 2}^{\infty} \frac{2( x_{n} - x_{n-1} ) x_{n}}{\Big(  1 + \big(  \frac{x_{n}}{\rho E}  \big)^{3}\Big)^{2}}\\
&\leq \frac{30}{\big( 1 + \frac{64}{\rho^{3} E^{3}} \big)^{2}} + 2 \rho^{2}E^{2}\int_{\frac{4}{\rho E}}^{\infty}\frac{x}{(1 + x^{3})^{2}} dx\\
&\doteq 2\rho^{2}E^{2}\int_{0}^{\infty}\frac{x}{(1 + x^{3})^{2}} dx + \delta(E),
\end{split}
\end{eqnarray}
where $$ x_{n} = n(n + 3), \ n = 1, 2, ...,  $$ and
$$  \delta(E) = \frac{30}{\big( 1 + \frac{64}{\rho^{3} E^{3}} \big)^{2}} - 2 \rho^{2}E^{2}\int_{0}^{\frac{4}{\rho E}}\frac{x}{(1 + x^{3})^{2}}.$$
By numerical calculation, for any $E \in [0, E_{0}]$, we obtain $ \delta(E) \leq 0. $ Then, combining with $\ref{S2161}$, the following inequality
\begin{eqnarray}\label{S220}
\frac{3(\rho E)^{4}}{2B(\frac{2}{3}, \frac{4}{3})} \sum_{n = 1}^{\infty} \frac{(2n + 3)\big( n(n + 3) + 2 \big)}{\big(  (\rho E)^{3} + \big( n(n + 3) \big)^{3}  \big)^{2}}\ \  \leq\ \  1
\end{eqnarray}
holds for any $E$ in interval $ [0, E_{0} ].$
Moreover, applying inequality  (\ref{S2161}) and (\ref{S216}) to (\ref{S214}), and combining (\ref{S220}) with (\ref{S219}), which gives
\begin{eqnarray}\label{S221}
\begin{split}
 \sum_{j = 1}^{N}| u_{j}(x) - u_{j}^{E}(x) |^{2} &\leq \frac{1}{6\omega_{4}}\sum_{n = 1}^{\infty} \frac{(2n + 3)\big(2n(n + 3)\big)}{\Big(  1 + \big(  \frac{n(n + 3)}{\rho E}  \big)^{3}\Big)^{2}}\\
 &\leq \frac{B(\frac{2}{3}, \frac{4}{3})\rho^{2}E^{2}}{9\omega_{4}}\\
 &\doteq C_{0}E^{2},
\end{split}
\end{eqnarray}
where $ C_{0} = \frac{B(\frac{2}{3}, \frac{4}{3})\rho^{2}}{9\omega_{4}}.$
Inserting (\ref{S221}) into   (\ref{S210}), and optimizing it with respect to $\epsilon$, we obtain
\begin{eqnarray}\label{S222}
\begin{split}
\sum_{j = 1}^{N}|u_{j}(x)|^{2} &\leq (1 + \epsilon) \sum_{j = 1}^{N}|u_{j}^{E}(x)|^{2} + (1 + \frac{1}{\epsilon})C_{0}E^{2}\\
&\leq \bigg(\sqrt{\sum_{j = 1}^{N}|u_{j}^{E}(x)|^{2}} + \sqrt{C_{0}E^{2}}\bigg)^{2}.
\end{split}
\end{eqnarray}
Solving $ \sum_{j = 1}^{N}|u_{j}^{E}(x)|^{2} $ from inequality (\ref{S222}), which yields
\begin{eqnarray}\label{S224}
\sum_{j = 1}^{N}|u_{j}^{E}(x)|^{2} \geq \bigg(  \sqrt{ \sum_{j = 1}^{N}|u_{j}(x)|^{2} }  - \sqrt{C_{0}E^{2}} \bigg)_{+}^{2}.
\end{eqnarray}
Equation (\ref{S28}) implies that
\begin{eqnarray}\label{S225}
\begin{split}
\sum_{j = 1}^{N}\int_{\mathbb{S}^{4}}|\nabla_{0} u_{j}(x) |^{2}dx = \int_{\mathbb{S}^{4}}\int_{0}^{\infty} \sum_{j = 1}^{N}|u_{j}^{E}(x)|^{2} dE dx.
\end{split}
\end{eqnarray}
Denote $ \phi(x) = \sum_{j = 1}^{N}|u_{j}(x)|^{2} ,$ and
inserting (\ref{S224}) into (\ref{S225}), we obtain
\begin{eqnarray*}
\begin{split}
\sum_{j = 1}^{N}\int_{\mathbb{S}^{4}}|\nabla_{0} u_{j}(x) |^{2}dx &\geq \int_{\mathbb{S}^{4}}\int_{0}^{\infty}  \bigg(  \sqrt{ \sum_{j = 1}^{N}|u_{i}(x)|^{2} }  - \sqrt{C_{0}E^{2}} \bigg)_{+}^{2} dE dx \\
&= \int_{\mathbb{S}^{4}}\int_{0}^{\sqrt{\frac{\phi(x)}{C_{0}}}}  \big(   \phi(x)   - 2\sqrt{\phi(x)C_{0}E^{2}} + C_{0}E^{2} \big) dE dx \\
&= \frac{1}{3\sqrt{C_{0}}}\int_{\mathbb{S}^{4}} \big( \phi(x)\big)^{ 1 + \frac{2}{4}} dx \\
&= \frac{1}{3\sqrt{\frac{B(\frac{2}{3}, \frac{4}{3})\rho^{2}}{9\omega_{4}}}}\int_{\mathbb{S}^{4}} \big( \sum_{j = 1}^{N}|u_{j}(x)|^{2}\big)^{ \frac{3}{2}} dx,
\end{split}
\end{eqnarray*}
where
$$ \omega_{4} = \frac{8 \pi^{2}}{3}, \ \ \ \   \rho = \frac{4\pi}{9\sqrt{3}}.$$
Thus,
$$ \sum_{j = 1}^{N}\int_{\mathbb{S}^{4}}|\nabla_{0} u_{j}(x) |^{2}dx \geq \frac{9}{\sqrt{2B(\frac{2}{3}, \frac{4}{3})}} \int_{\mathbb{S}^{4}} \big( \sum_{j = 1}^{N}|u_{j}(x)|^{2}\big)^{ \frac{3}{2}}dx, $$
namely,
$$ \mathcal{K}_{4}(\mathbb{S}^{4}) \leq \frac{\sqrt{2 B(\frac{2}{3}, \frac{4}{3})}}{9} = 0.1728.$$
This complete the proof the theorem.
\end{proof}

\section{Lieb-Thirring Inequality on $ \mathbb{T}^{4} $}
In this section, we turn to study the constant of Lieb-Thirring inequality on $ \mathbb{T}^{4}.$
Without loss of generlity, let $ \mathbb{T}^{4} = [0, 2\pi]^{4}.$ First, let's introduce some useful known conclusions on torus \cite{L14}.
The series $ \{ 4\pi^{-2}e^{im\cdot x} \}_{m}, m\in \mathbb{Z}_{0}^{4} = \mathbb{Z}^{4}\backslash 0$ is an orthonormal basis in $L^{2}(\mathbb{T}^{4})$, where $ \mathbb{Z} $ denotes the set of non-negative integers.
In the reminder of this section, $ \nabla $ denotes the gradient operator on the Euclidean space.
The following theorem is about the Lieb-Thirring inequalities on the four-dimensional torus $\mathbb{S}^{4}$, namely, the Theorem $1.1$ for $ \mathbb{T}^{4}$.
\begin{thm}\label{thm2}
Let $u_{1}, u_{2}, ..., u_{N} \in \{u | \int_{\mathbb{T}^{4}} u dx = 0, u \in H^{1}(\mathbb{T}^{4})  \} $ be othornormal in $ L^{2}(\mathbb{T}^{d}) $. Then
\begin{equation}\label{thm2LT}
\int_{\mathbb{T}^{4}}\big( \sum_{j = 1}^{N}|u_{j}(x)|^{2}  \big)^{\frac{3}{2}}dx \leq \mathcal{K}_{4}(\mathbb{T}^{4}) \sum_{i = 1}^{N}\int_{\mathbb{T}^{4}}|\nabla u_{j} |^{2}dx,
\end{equation}
where the constant $ \mathcal{K}_{4}(\mathbb{T}^{4}) $ satisfies
\begin{eqnarray}\label{thm2LT1}
0.0190 \leq \mathcal{K}_{4}(\mathbb{T}^{4}) \leq 0.1222.
\end{eqnarray}
\end{thm}
\begin{proof}
We first prove the left side of the inequality $ \ref{thm2LT1} $, namely, $ \mathcal{K}_{4}(\mathbb{T}^{4}) \geq \frac{3}{16\pi^{2}}.$
Follows \cite{L14}, $ \{ \frac{e^{im\cdot x}}{4\pi^{2}} \}_{m\in \mathbb{Z}^{4}} $ is an orthonormal basis of $ L^{2}(\mathbb{T}^{4}),$ where $ m = (m_{1}, m_{2}, m_{3}, m_{4})$ with $ |m| = \sqrt{m_{1}^{2} + m_{2}^{2} + m_{3}^{2} + m_{4}^{2} }.$
Fixed $ M > 0 $, taking $ \{ u_{j} \}_{j} $ for the  orthonormal series $ \{ \frac{e^{im\cdot x}}{4\pi^{2}} \}_{m\in \mathbb{Z}^{4}}$, where $ m $ satisfies  $ 0 \leq m_{k} \leq M(k = 1, 2, 3, 4), $ so that
\begin{eqnarray}\label{S3a1}
\begin{split}
\int_{\mathbb{T}^{4}} \sum_{j}u_{j}^{2}(x)dx &=  \int_{\mathbb{T}^{4}} \sum_{0 \leq m_{k} \leq M}\big( \frac{e^{im\cdot x}}{4\pi^{2}} \big)^{2}(x) = ( M + 1)^{4}\\
&\leq (2\pi)^{\frac{4}{3}}\Big(  \int_{\mathbb{T}^{4}}\big(  \sum_{0 \leq m_{k} \leq M}( \frac{e^{im\cdot x}}{4\pi^{2}})^{2}(x)  \big)^{\frac{3}{2}}  \Big)^{\frac{2}{3}}.
\end{split}
\end{eqnarray}
Then, applying Lieb-Thirring inequality (\ref{thm2LT}) to the last term of (\ref{S3a1}), we have
\begin{eqnarray}\label{S3a2}
\begin{split}
\int_{\mathbb{T}^{4}}\big(  \sum_{0 \leq m_{k} \leq M}( \frac{e^{im\cdot x}}{4\pi^{2}})^{2}(x)  \big)^{\frac{3}{2}} &\leq \mathcal{K}_{4}(\mathbb{T}^{4}) \sum_{0 \leq m_{k} \leq M}\int_{\mathbb{T}^{4}}|\nabla \frac{e^{im\cdot x}}{4\pi^{2}} |^{2}dx\\
&\leq \mathcal{K}_{4}(\mathbb{T}^{4})\sum_{t = 1}^{M}4(M + 1)^{3}t^{2}\\
&= \frac{2}{3}M(M + 1)^{4}(2M + 1)\mathcal{K}_{4}(\mathbb{T}^{4}).
\end{split}
\end{eqnarray}
Inserting inequality (\ref{S3a2}) into inequality (\ref{S3a1}), and taking both sides  to the 2/3 power, which gives
\begin{eqnarray}\label{S3a3}
(M + 1)^{6} \leq \frac{2}{3}(2\pi^{2})M(M + 1)^{4}(2M + 1)\mathcal{K}_{4}(\mathbb{T}^{4}).
\end{eqnarray}
Taking $ M $ tends to $ \infty $, we obtain
$$ \mathcal{K}_{4}(\mathbb{T}^{4}) \geq \frac{3}{16\pi^{2}} = 0.0190.$$

Next, we turn to prove the right side of inequality (\ref{thm2LT1}). For any $u_{j} \in L^{2}(\mathbb{T}^{4})$, the Fourier analysis on torus (Chapter $3$ in  \cite{L14}) shows that
\begin{eqnarray}\label{S31}
u_{j}(x) = \frac{1}{4\pi^{2}}\sum_{k\in Z_{0}^{4}} u_{j}^{k} e^{ik\cdot x}, \ \ \ u_{j}^{k} = \frac{1}{4\pi^{2}} \int_{\mathbb{T}^{4}}  u_{j}(x) e^{-ik\cdot x} dx,
\end{eqnarray}
where $ \mathbb{Z}_{0} = \mathbb{Z}^{4}\backslash \{0\},$
with $ z \in \mathbb{Z}_{0} $
$$ z = (z_{1}, z_{2}, z_{3}, z_{4})\ \  \text{and}\ \ |z| = \sqrt{z_{1}^{2} + z_{2}^{2} + z_{3}^{2} + z_{4}^{2} }.$$
Using (\ref{S31}), we have
\begin{eqnarray}\label{S32}
\| u_{j}(x) \|^{2} = \int_{\mathbb{T}^{4}} | u_{j}(x) |^{2} dx = \sum_{k \in \mathbb{Z}_{0}^{4}} | u_{j}^{k} |^{2},
\end{eqnarray}
and from (\ref{S32}) the gradient of $ u_{j}(x) $:
\begin{eqnarray}\label{S33}
\| \nabla u_{j}(x) \|^{2} = \int_{\mathbb{T}^{4}} | \nabla u_{j}(x) |^{2} dx = \sum_{ k \in \mathbb{Z}_{0}^{4} } | k |^{2} |u_{j}^{k}|^{2} .
\end{eqnarray}
Togather (\ref{S32}) with (\ref{S21}), and applying the calculation methods from (\ref{S28}) to (\ref{S33}), we obtain
\begin{eqnarray}\label{S34}
\begin{split}
\| \nabla u_{j}(x) \|^{2} &= \int_{0}^{\infty} \sum_{k \in \mathbb{Z}_{0}^{4} } \big( g( \frac{E}{|k|^{2}} ) \big)^{2}| u_{j}^{k} |^{2} dE \\
&\doteq \int_{ \mathbb{T}^{4} } \int_{0}^{\infty}| u_{j}^{E}(x) |^{2} dE dx,
\end{split}
\end{eqnarray}
where
\begin{equation}\label{S35}
 u_{j}^{E}(x)  = \frac{1}{4 \pi^{2}} \sum_{ k\in \mathbb{Z}_{0}^{4} } g( \frac{E}{|k|^{2}} ) u_{j}^{k} e^{ik\cdot x}.
\end{equation}
(\ref{S31}) and (\ref{S35}) implies that
\begin{eqnarray}\label{S36}
\begin{split}
\sum_{j = 1}^{N}|u_{j}(x) - u_{j}^{E}(x)|^{2} &= \sum_{j = 1}^{N}| \frac{1}{4\pi^{2}} \sum_{k \in \mathbb{Z}_{0}^{4} } \big(1 - g(\frac{E}{| k |^{2}})  \big)u_{j}^{k}e^{ik \cdot x}|^{2} \\
&\doteq \sum_{j = 1}^{N}| \big( u_{j}(x'),  \Phi^{E}(x', x) \big)|^{2},
\end{split}
\end{eqnarray}
where
$$ \Phi^{E}(x', x) = \frac{1}{4 \pi^{2}} \sum_{ k\in \mathbb{Z}_{0}^{4} }\big( 1 - g( \frac{E}{|k|^{2}} )\big)e^{ik\cdot x'} e^{-ik\cdot x},  $$
and
$$ \big( u_{j}(x'),  \Phi^{E}(x', x) \big) \doteq \int_{\mathbb{T}^{4}} u_{j}(x')\overline{\Phi^{E}}(x', x)dx'.$$
Since $ \{u_{j}\} $ is orthonormal series in $ H^{1}(\mathbb{T}^{4}),$ by Bessel's inequality
\begin{eqnarray}\label{S37}
\begin{split}
\sum_{j = 1}^{N}|u_{j}(x) - u_{j}^{E}(x)|^{2} &= \sum_{j = 1}^{N}| \big( u_{j}(x'),  \Phi^{E}(x', x) \big)|^{2}\\
&\leq \| \Phi^{E}(x', x) \|^{2}\\
&= \frac{1}{16\pi^{4}} \sum_{ k\in \mathbb{Z}_{0}^{4} }\big( 1 - g( \frac{E}{|k|^{2}} )\big)^{2}.
\end{split}
\end{eqnarray}
Next, the key step is to estimate  $ \| \Phi^{E}(x', x) \|^{2},$ namely, to estimate the sum of the series $\{ \big( 1 - g( \frac{E}{|k|^{2}} )\big)^{2} \}_{k\in \mathbb{Z}_{0}^{4} }.$
Before that, we have to determine the concrete expression of $ g(t) $ in (\ref{S21}).
Using the Lemma $ 5 $ in \cite{FHJN21} again, we  obtain the expression of $g$:
\begin{eqnarray}\label{S38}
g(t) = \frac{1}{1 + {(\mu t)}^{3}},
\end{eqnarray}
where
$$ \mu  =   \frac{4\pi}{9\sqrt{3}}.$$
Inserting (\ref{S38}) into (\ref{S37}), we have
\begin{eqnarray}\label{S39}
\begin{split}
\| \Phi^{E}(x', x) \|^{2} &= \frac{1}{16\pi^{4}} \sum_{ k\in \mathbb{Z}_{0}^{4} }\big( 1 - g( \frac{E}{|k|^{2}} )\big)^{2} \\
&= \frac{1}{16\pi^{4}} \sum_{ k\in \mathbb{Z}_{0}^{4} } \Big(  1 - \frac{1}{1 + \big(\frac{\mu E}{|k|^{2}}\big)^{3}}     \Big)^{2} \\
&= \frac{1}{16\pi^{4}} \sum_{ k\in \mathbb{Z}_{0}^{4} } \frac{1}{\Big(  1 + \big(\frac{|k|^{2}}{\mu E}\big)^{3}     \Big)^{2}}\\
&\doteq \frac{1}{16\pi^{4}} \sum_{ k\in \mathbb{Z}_{0}^{4} }\varphi(\frac{k}{\nu}),
\end{split}
\end{eqnarray}
where $ \nu = \sqrt{\mu E} $ and
$$ \varphi(x) = \frac{1}{\big(1 +   |x|^{6}  \big)^{2}}, \ \ x \in \mathbb{R}^{4} $$
with $ | x |= \sqrt{x_{1}^{2} + x_{2}^{2} + x_{3}^{2} + x_{4}^{2} }.$
From the expression of the function $\varphi$, $\varphi(x) \in L^{2}(\mathbb{R}^{4}).$
Therefore, the Fourier Transform of $\hat{\varphi}$ is defined, namely,
$$ \hat{\varphi}(\xi) \doteq \frac{1}{4\pi^{2}}\int_{\mathbb{R}^{4}} \varphi(x)e^{-i\xi\cdot x}dx. $$
By the Possion summation formula (Chapter 3 in \cite{L14}):
$$ \sum_{m\in \mathbb{Z}^{4}} f(m/\nu) = (2\pi^{2})\nu^{4} \sum_{m\in \mathbb{Z}^{4}}\hat{f}(2\pi m \nu), $$
then, the right side of (\ref{S39}) can be reformulated as follows:
\begin{eqnarray}\label{S310}
\begin{split}
\sum_{x\in \mathbb{Z}_{0}^{4}}\varphi(\frac{k}{\nu}) &= (2\pi)^{2}\nu^{4}\sum_{k\in \mathbb{Z}^{4}} \hat{\varphi}(2\pi |k|\nu) - 1 \\
&= \nu^{4}\int_{\mathbb{R}^{4}} \varphi(x)dx + (2\pi)^{2}\nu^{4}\sum_{k\in \mathbb{Z}_{0}^{4}} \hat{\varphi}(2\pi |k|\nu) - 1 \\
&= \nu^{4}\int_{\mathbb{R}^{4}} \varphi(x)dx + O(e^{-C\nu}) - 1,
\end{split}
\end{eqnarray}
where $C$ is a positive constant. Thus, there exists a positive constant $ \nu_{0} \  (0 <  \nu_{0} <  \infty) $ such that
\begin{eqnarray}\label{S311}
\begin{split}
\sum_{x\in \mathbb{Z}_{0}^{4}}\varphi(\frac{k}{\nu}) &\leq \nu^{4}\int_{\mathbb{R}^{4}} \varphi(x)dx\\
&= \frac{\omega_{3}B(\frac{2}{3}, \frac{4}{3})}{6}\nu^{4},
\end{split}
\end{eqnarray}
holds on for $ \nu > \nu_{0}. $
On the other hand, the analysis in section 4.2 of \cite{I12} shows that the inequality (\ref{S311}) also holds on  for any $ \nu  \in [0, \nu_{0} ].$
Combining (\ref{S311}) and (\ref{S39}), we can finish the estimation of $ \| \Phi^{E}(x', x) \|^{2}, $ namely,
\begin{eqnarray}\label{S313}
\| \Phi^{E}(x', x) \|^{2}  \leq \frac{\omega_{3}B(\frac{2}{3}, \frac{4}{3})}{96 \pi^{4}}\nu^{4}= \frac{\mu^{2} \omega_{3}B(\frac{2}{3}, \frac{4}{3})}{96 \pi^{4}} E^{2} \doteq \bar{C}E^{2},
\end{eqnarray}
where $ \bar{C} = \frac{\mu^{2} \omega_{3}B(\frac{2}{3}, \frac{4}{3})}{96 \pi^{4}}.$
We now return to  the series $ \{ u_{j}(x)\} $, for any $ \varepsilon > 0 $, we have
\begin{eqnarray}\label{S312}
\begin{split}
\sum_{j = 1}^{N}|u_{j}(x)|^{2} &= \sum_{j = 1}^{N}\big(| u_{j}(x) - u_{j}^{E}(x) | +   |u_{j}^{E}(x)|  \big)^{2} \\
&\leq (1 + \varepsilon)\sum_{j = 1}^{N}|u_{j}^{E}(x)|^{2} + (1 + \frac{1}{\varepsilon})\sum_{j = 1}^{N}| u_{j}(x) - u_{j}^{E}(x) |^{2}. \\
\end{split}
\end{eqnarray}
Inserting (\ref{S37}) and (\ref{S313}) into (\ref{S312}), and optimizing it with  respect to $\varepsilon$, we obtain
\begin{eqnarray}\label{S314}
\begin{split}
\sum_{j = 1}^{N}|u_{j}(x)|^{2} &\leq (1 + \varepsilon)\sum_{j = 1}^{N}|u_{j}^{E}(x)|^{2} + (1 + \frac{1}{\varepsilon})\sum_{j = 1}^{N}| u_{j}(x) - u_{j}^{E}(x) |^{2}\\
&\leq (1 + \varepsilon)\sum_{j = 1}^{N}|u_{j}^{E}(x)|^{2} + (1 + \frac{1}{\varepsilon})\bar{C}E^{2}\\
&\leq \Big( \sqrt{\sum_{j = 1}^{N}|u_{j}^{E}(x)|^{2}}  + \sqrt{\bar{C}E^{2}}   \Big)^{2}
\end{split}
\end{eqnarray}
Solving $ \sum_{j = 1}^{N}|u_{j}^{E}(x)|^{2} $ from inequality (\ref{S314}), we have
\begin{eqnarray}\label{S315}
\sum_{j = 1}^{N}|u_{j}^{E}(x)|^{2} \geq \Big( \sqrt{\sum_{j = 1}^{N}|u_{j}(x)|^{2}} - \sqrt{\bar{C}E^{2}} \Big)_{+}^{2}.
\end{eqnarray}
Denote $ \phi(x) = \sum_{j = 1}^{N}|u_{j}(x)|^{2} ,$ summing  both sides of the equation (\ref{S34}) with respect to $j$, and together with (\ref{S315}) yields
\begin{eqnarray}\label{S316}
\begin{split}
\sum_{j = 1}^{N} \| \nabla u_{j}(x) \|^{2} &= \int_{ \mathbb{T}^{4} } \int_{0}^{\infty} \sum_{j = 1}^{N} |u_{j}^{E}(x) |^{2} dE dx\\
&\geq \int_{ \mathbb{T}^{4} } \int_{0}^{\infty}\Big( \sqrt{\sum_{j = 1}^{N}|u_{j}(x)|^{2}} - \sqrt{\bar{C}E^{2}} \Big)_{+}^{2}dE dx\\
&= \int_{ \mathbb{T}^{4} } \int_{0}^{\sqrt{\frac{\phi(x)}{\bar{C}}}}\Big( \phi(x) -2\sqrt{\bar{C} \phi(x)E^{2}} + \bar{C}E^{2} \Big)^{2}dE dx\\
&= \frac{1}{3\sqrt{\bar{C}}}\int_{ \mathbb{T}^{4} } \big( \phi(x)\big)^{1 + \frac{2}{4}} dx\\
&= \frac{1}{3\sqrt{\frac{\mu^{2} \omega_{3}B(\frac{2}{3}, \frac{4}{3})}{96 \pi^{4}}}}\int_{ \mathbb{T}^{4} } \big( \sum_{j = 1}^{N}|u_{j}(x)|^{2}\big)^{1 + \frac{2}{4}}dx,
\end{split}
\end{eqnarray}
where
$$ \omega_{3} = 2\pi^{2},\ \ \ \mu = \frac{4\pi}{9\sqrt{3}}.$$
Thus,
$$ \sum_{j = 1}^{N} \| \nabla u_{j}(x) \|^{2} \geq \frac{9}{\sqrt{B(\frac{2}{3}, \frac{4}{3})}}\int_{ \mathbb{T}^{4} } \big( \sum_{j = 1}^{N}|u_{j}(x)|^{2}\big)^{ \frac{3}{2}}dx $$
So we obtain
$$ \mathcal{K}_{4}(\mathbb{T}^{4}) \leq \frac{\sqrt{B(\frac{2}{3}, \frac{4}{3})}}{9} = 0.1222. $$
This completes the proof of the theorem.
\end{proof}

\end{document}